\documentclass[a4papersize]{article}

\usepackage{amsmath,amsthm, amsxtra,amssymb,latexsym, amscd}
\usepackage[mathscr]{eucal}
\newtheorem{theorem}{Theorem}[section]
\newtheorem{corollary}[theorem]{Corollary}
\newtheorem{lemma}[theorem]{Lemma}
\newtheorem{proposition}[theorem]{Proposition}
\newtheorem{example}[theorem]{Example}

\DeclareMathOperator{\depth}{depth}

\DeclareMathOperator{\Hom}{Hom}
\DeclareMathOperator{\Ann}{Ann}
\DeclareMathOperator{\docao}{ht}
\DeclareMathOperator{\Ass}{Ass}
\DeclareMathOperator{\Supp}{Supp}
\DeclareMathOperator{\Ext}{Ext}
\DeclareMathOperator{\Var}{Var}
\DeclareMathOperator{\psd}{psd}
\DeclareMathOperator{\Psupp}{Psupp}
\DeclareMathOperator{\nCM}{nCM}

\DeclareMathOperator{\Att}{Att}

\DeclareMathOperator{\Spec}{Spec}
\DeclareMathOperator{\p}{\frak p}
\DeclareMathOperator{\q}{\frak q}
\DeclareMathOperator{\m}{\frak m}
\DeclareMathOperator{\R}{\widehat R}
\DeclareMathOperator{\lr}{\longrightarrow}
\begin{document}
\large
\centerline{\Large {\bf  ON PSEUDO SUPPORTS AND NON COHEN-MACAULAY }}
\medskip

\centerline{\Large {\bf  LOCUS OF FINITELY GENERATED MODULES}}

\vskip 0.7cm
\centerline {NGUYEN TU CUONG}
\centerline{ Institute of Mathematics}
\centerline{18 Hoang Quoc Viet Road, 10307 Hanoi, Vietnam}
\centerline{ E-mail: ntcuong@math.ac.vn}
\vskip 0.2cm
\centerline { LE THANH NHAN}
\centerline{Thai Nguyen College of Sciences, Thai Nguyen, Vietnam}
\centerline{E-mail: trtrnhan@yahoo.com}
\vskip 0.2cm
\centerline { NGUYEN THI KIEU NGA}
\centerline{Hanoi Pedagogical University N$^0$2, Vinh Phuc, Vietnam}
\centerline{E-mail: kieungasp2@gmail.com}
\vskip 1cm

\noindent{\bf Abstract} {\footnote{ {\it{Key words and phrases: }} Pseudo supports, non Cohen-Macaulay locus, catenarity, Serre conditions, unmixedness. \hfill\break
  {\it{2000 Subject  Classification: }} 13D45, 13E05.\hfill\break {The authors are supported by the Vietnam National Foundation for Science and Technology Development (Nafosted).}}. { Let $(R,\m )$ be a Noetherian local ring  and $M$ a finitely generated $R$-module with $\dim M=d.$  Let  $i\geq 0$ be an integer. Following M. Brodmann and R. Y. Sharp \cite{BS1}, the $i$-th pseudo support of $M$  is the set of all prime ideals $\p$ of $R$ such that $ H^{i-\dim (R/\p)}_{\p R_{\p}}(M_{\p})\neq 0.$  In this paper, we study the pseudo supports  and the non Cohen-Macaulay locus of $M$ in connections with the catenarity of the ring $R/\Ann_RM$, the Serre conditions on $M$, and the unmixedness of the local rings $R/\p$ for certain prime ideals $\p$ in $\Supp_R (M)$.
 \section{Introduction}
 Throughout this paper, let $(R,\m )$ be a Noetherian local ring  and $M$ a finitely generated $R$-module with $\dim M=d.$  For each ideal $I$ of $R$, denote by $\Var (I)$ the set of all prime ideals containing $I$.  Let  $i\geq 0$ be an integer.  Following M. Brodmann and R. Y. Sharp [BS1], {\it the $i$-th pseudo support} of $M$, denoted by $\Psupp_R^iM$, is defined  by 
$$\Psupp_R^i(M)=\{\p\in\Spec (R) \mid  H^{i-\dim (R/\p)}_{\p R_{\p}}(M_{\p})\neq 0\}.$$
 Suppose that $R$ is a quotient of a $d'$-dimensional Gorenstein local  ring $(R', \m').$ Denote by $K^i_M$ the $R$-module $\Ext^{d'-i}_{R'}(M,R').$ Then $K^i_M$ is a finitely generated $R$-module and the local duality gives  an isomorphism $H^i_{\m}(M)\cong \Hom_R(K^i_M, E(R/\m))$, where $E(R/\m)$ is the injective hull of $R/\m,$  cf. \cite[11.2.6]{BS}.  This isomorphism  was used  to prove the closedness of  the non Cohen-Macaulay  locus $\nCM (M)$  of $M$ which is defined by 
$$\nCM (M)=\{\p\in\Spec (R)\mid M_{\p}\  \text{is not  Cohen-Macaulay}\},$$
 cf.  \cite[Corollary 6]{Sch1}.  Moreover we have $$\Psupp^i_R(M)=\Var (\Ann_R(H^i_{\m}(M)))=\Supp_R(K^i_M),$$ therefore $\Psupp^i_R(M)$  is a closed subset of $\Spec (R)$. Ones can use this fact in conjunction with the classical associativity formula for multiplicity of the finitely generated $R$-module $K^i_M$  to produce analogous associativity formula for multiplicity of $H^i_{\m}(M).$ In case $R$ is universally catenary and all its formal fibres are Cohen-Macaulay,  the associativity formula for multiplicity of $H^i_{\m}(M)$ is still valid and  $\Psupp^i_R(M)=\Var (\Ann_R(H^i_{\m}(M)))$, a closed subset of $\Spec (R)$ (cf. \cite[Theorem 2.4, Proposition 2.5]{BS1}). 

 The purpose of this paper is to study the pseudo supports and  the non Cohen-Macaulay locus of $M$ in connections with the catenarity of the ring $R/\Ann_RM$, the Serre conditions on $M$ and the unmixedness of the local rings $R/\p$ for certain prime ideals $\p$ in $\Supp_R(M)$.  The results in this paper show that, even without any assumption on $R$ and  even  $\Psupp^i_R(M)$ may not be closed,  the pseudo supports  of $M$ still give a lot of useful information on the module $M$ and the base ring $R$. 

The paper is divided into 4 sections. In the next section, we present some basic properties of pseudo supports of $M$ which will be used in the sequel.  In Section 3,  we provide a description of the depth and dimension of the localizations of $M$ and obtain a formula for the non Cohen-Macaulay locus  of $M$ (Theorem \ref{T:1a}). It follows that if all pseudo supports of $M$ are closed then so is the non Cohen-Macaulay locus of $M$ (Corollary \ref{C:1}).  Some further relations between the closedness of $\Psupp^i_R(M)$ and  that of $\nCM(M)$ are also given. In the last section, firstly we  show that many results already known under the assumption that $R$ is a quotient of a Gorenstein ring, can be extended to the case where $R/\Ann_RM$ is universally catenary and all its formal fibres are Cohen-Macaulay. Especially the non Cohen-Macaulay locus of $M$ is closed under this weaker hypothesis (Corollary \ref{C:6}). Then we present some connections with the universal catenarity of the ring $R/\Ann_RM$, the  Serre conditions on $M$  and the unmixedness of the rings $R/\p$ for certain $\p\in\Supp_R(M)$ (Theorems \ref{T:2}, \ref{T:3}).

\section {Preliminaries}  

Let $i\geq 0$ be an integer. Recall that {\it the $i$-th pseudo support} of $M$, denoted by $\Psupp_R^iM$, is defined  as follows, cf. \cite{BS1}
$$\Psupp_R^i(M)=\{\p\in\Spec (R) \mid  H^{i-\dim (R/\p)}_{\p R_{\p}}(M_{\p})\neq 0\}.$$
For  a subset $T$ of $\Spec (R)$,  we set $(T)_i=\{\p\in T\mid \dim (R/\p)=i\}.$
\begin{lemma}\label{L:1d} Let $i\geq 0$ be an integer.  The following statements are true

(i) $\dim (R/\p)\leqslant i$ for all $\p\in\Psupp^i_R(M)$. 

(ii) $\big(\Psupp^i_R(M)\big)_i=\big(\Ass_RM\big)_i.$
\end{lemma}

\begin{proof} (i). Let $\p\in \Psupp^i_R(M)$.  Then $H^{i-\dim (R/\p)}_{\p R_{\p}}(M_{\p})\neq 0.$ Therefore $i\geq \dim (R/\p).$

\noindent (ii).  It is clear that $\p\in \big(\Psupp^i_R(M)\big)_i$ if and only if $H^0_{\p R_{\p}}(M_{\p})\neq 0$ and $\dim (R/\p)=i$, if and only if  $\p R_{\p}\in\Ass_{R_{\p}}(M_{\p})$ and $\dim (R/\p)=i$, if and only if  $\p\in \big(\Ass_RM\big)_i.$  
\end{proof}

Next we  give a relation between $\Psupp^i_RM$ and $\Var (\Ann_RH^i_{\m}(M))$. Before doing that, we need some facts on the secondary representation theory of Artinian modules: Following I. G. Macdonald [Mac], every Artinian $R$-module $A$ has a minimal secondary representation $A=A_1+\ldots +A_n,$ where $A_i$ is $\p_i$-secondary. The set $\{\p_1,\ldots ,\p_n\}$ is independent of the choice of the minimal  secondary representation of $A$. This set is called {\it the set of attached prime ideals}   of $A,$ and denoted by  $\Att_RA$.   
\begin{lemma} \cite{Mac}\label{L:1a} Let $A$ be an Artinian $R$-module. Then  $A\neq 0$ if and only if $\Att_RA\neq \emptyset.$ Moreover, the set of all minimal elements of $\Att_RA$ is exactly the set of all minimal elements of $\Var (\Ann_RA)$. 
\end{lemma}

 \begin{lemma}\label{L:1b} Let $i\geq 0$ be an integer.  Then $\Psupp^i_R(M)\subseteq\Var (\Ann_RH^i_{\m}(M))$. 
 \end{lemma}
\begin{proof}  Suppose that  $\p \in\Psupp^i_R (M)$.  Then $H^{i-\dim (R/\p)}_{\p R_{\p}}(M_{\p})\neq 0.$ Since $H^{i-\dim (R/\p)}_{\p R_{\p}}(M_{\p})$ is an Artinian $R_{\p}$-module, there exists by Lemma \ref{L:1a} a prime ideal $\q$ contained in $\p$  such that  $\q R_{\p}\in\Att_{R_{\p}}\big(H^{i-\dim (R/\p)}_{\p R_{\p}}(M_{\p})\big).$  Therefore we get by Weak General Shifted Localization Principle [BS, 11.3.8] that $\q\in\Att_R(H^i_{\m}(M)).$ Hence $\q\supseteq \Ann_R(H^i_{\m}(M))$ by Lemma \ref{L:1a}, and hence $\p\supseteq \Ann_R(H^i_{\m}(M))$. Therefore   $\Psupp^i_R (M)\subseteq \Var (\Ann_RH^i_{\m}(M))$.
\end{proof}

Following M. Brodmann and R. Y. Sharp [BS1], the {\it $i$-th pseudo dimension} of $M$, denoted by $\psd^iM$, is defined by
$$\psd^i (M)=\max\{\dim (R/\p) \mid \p\in\Psupp^i_R(M)\}.$$
So, the module $H^i_{\m}(M)$ is related to the four notions of dimension: $\dim (R/\Ann_RH^i_{\m}(M))$, $\dim (\R/\Ann_{\R}H^i_{\m}(M)),$ $\psd^i(M)$ and $\psd^i(\widehat M).$ Below we compare these notions.

\begin{proposition} \label{P:1d} Let $i\geq 0$ be an integer.  Then
$$\psd^i(M)\leqslant \psd^i(\widehat M)=\dim (\R/\Ann_{\R}H^i_{\m}(M))\leqslant \dim (R/\Ann_RH^i_{\m}(M)).$$
\end{proposition}

 \begin{proof} Let $\p\in\Psupp^i_R(M)$ such that $\psd^i(M)=\dim (R/\p).$  Then $H^{i-\dim (R/\p)}_{\p R_{\p}}(M_{\p})\neq 0.$ Let $\widehat\p\in\Ass (\R/\p\R)$ such that $\dim (R/\p)=\dim (\R/\widehat\p ).$ From the  map $R_{\p}\lr \R_{\widehat\p}$ which is faithful flat, we get by the Base Change Theorem \cite[Theorem 4.3.2]{BS} that $$H^{i-\dim (\R/\widehat\p)}_{\widehat\p\R_{\widehat\p}}(\widehat M_{\widehat\p})\cong H^{i-\dim (R/\p)}_{\p R_{\p}}(M_{\p})\otimes \R_{\widehat\p}\neq 0.$$
It follows that $\widehat\p\in\Psupp^i_{\R}(\widehat M).$ Therefore $\psd^i(\widehat M)\geq \dim (\R/\widehat\p )=\psd^i(M).$ 
It is not difficult to check that $$\Psupp^i_{\R}(\widehat M)= \Var (\Ann_{\R}H^i_{\m\R}(\widehat M))=\Var (\Ann_{\R}H^i_{\m}(M)).$$ Therefore $\psd^i(\widehat M)=\dim (\R/\Ann_{\R}H^i_{\m}(M)).$ For the last inequality, by Lemma \ref{L:1a} there exists  $\widehat\q\in\min \Att_{\R}(H^i_{\m}(M))$ such that  
 $\dim (\R/\widehat\q)=\dim (\R/\Ann_{\R}H^i_{\m}(M)).$  By \cite[8.2.4, 8.2.5]{BS} we get $\widehat\q\cap R\in\Att_R(H^i_{\m}(M))$. So we have by Lemma \ref{L:1a} that
$$\dim (\R/\Ann_{\R}H^i_{\m}(M))=\dim (\R/\widehat\q)\leqslant \dim (R/(\widehat\q\cap R))\leqslant \dim (R/\Ann_RH^i_{\m}(M)).$$
\end{proof}
 It may happen that $\Psupp^i_R(M)$ is a proper subset of $\Var (\Ann_RH^i_{\m}(M))$ and $$\psd^i(M)< \dim (\R/\Ann_{\R}H^i_{\m}(M))<\dim (R/\Ann_RH^i_{\m}(M)).$$ Here is an example.

\begin{example} {\rm  (i). Let $(R,\m)$ be the Noetherian local domain of dimension $2$ constructed by D. Ferrand and M. Raynaud [FR] such that $\dim (\R/\widehat\q)=1$ for some $\widehat\q\in\Ass(\R).$ Then $\Psupp^1(R)=\{\m\}$ and hence $\psd^1(R)=0.$ Moreover  we have  $\dim (\R/\Ann_{\R}H^1_{\m}(R))=1$ and $\dim (R/\Ann_RH^1_{\m}(R))=2,$ cf. \cite[Example 4.1]{CN}. 

\noindent (ii). Let $(R,\m)$ be a Noetherian local domain of dimension $3$ such that $R$ is not catenary. By the similar arguments as in the proof of \cite[Proposition 4.6]{CDN} we can check that $\dim (\R/\Ann_{\R}H^2_{\m}(R))=2$ and $\dim (R/\Ann_RH^2_{\m}(R))=3.$  Since $R$ is not catenary, the set $U=\{\p\in\Spec (R)\mid \dim (R/\p)+\docao (\p)=2\}$ is not empty.  It is clear that $\p\in\Psupp^2(R)$ for all $\p\in U$ and $\dim (R/\p)\leqslant 1$ for all $\p\in\Psupp^2(R).$ Therefore $\psd^2(R)=1$.}
\end{example}

\section{Pseudo supports and non Cohen-Macaulay locus}

   Recall that  {\it the non Cohen-Macaulay locus} of $M$, denoted by $\nCM (M)$, is defined by
$$\nCM (M)=\{\p\in\Spec (R)\mid M_{\p}\ \text{is not Cohen-Macaulay}\}.$$

\begin{theorem}\label{T:1a} Suppose that $\p\in\Supp_R(M)$.  Then

(i)  $\p\in\Psupp^i_R(M)$ for some $i\leqslant d$ and   $$\depth (M_{\p})=k-\dim (R/\p),\ \dim (M_{\p})=t-\dim (R/\p),$$ where  $k=\underset{i\leqslant d}{\min}\{i\mid \frak \p\in\Psupp^i_R(M)\}$ and $t=\underset{i\leqslant d}{\max}\{i\mid \p\in\Psupp^i_R(M)\}.$ 

(ii) $\displaystyle \nCM (M)=\bigcup_{0\leqslant i<j\leqslant d}\big(\Psupp^i_R(M)\cap \Psupp^j_R(M)\big).$

(iii) If  $s\leqslant d$ is an integer then $$\bigcup_{i\leqslant s}\Psupp^i_R(M)=\{\p\in\Supp_R(M)\mid \depth (M_{\p})+\dim (R/\p )\leqslant s\}.$$

(iv)   If  $\displaystyle \p\notin \bigcup_{i<d}\Psupp^i_R(M)$  then $M_{\p}$ is Cohen-Macaulay of dimension $d-\dim (R/\p ).$ 
\end{theorem}

\begin{proof} (i). Since $\p\in\Supp_R(M)$, we have $M_{\p}\neq 0.$ Set $\dim M_{\p}=n.$ Then $n\geq 0$. Hence $H^n_{\p R_{\p}}(M_{\p})\neq 0.$ Since $n+\dim (R/\p)\leqslant d,$ we have $n=i-\dim (R/\p)$ for some $i\leqslant d.$  Therefore $H^{i-\dim (R/\p )}_{\p R_{\p}}(M_{\p})\neq 0,$ i.e. $\p\in\Psupp^i_R(M)$ for some $i\leqslant d.$  

Set $k=\underset{i\leqslant d}{\min}\{i\mid \p\in\Psupp^i(M)\}$.  Then $\p\in\Psupp^k_R(M)$, so $H^{k-\dim (R/\p)}_{\p R_{\p}}(M_{\p})\neq 0$.  Since $\p\notin\Psupp^i_R(M)$ for all $i<k$, we have  $H^{i-\dim (R/\p)}_{\p R_{\p}}(M_{\p})=0$ for all $i<k.$ Therefore $\depth (M_{\p})=k-\dim (R/\p ).$ Set $t=\underset{i\leqslant d}{\max}\{i\mid \p\in\Psupp^i_R(M)\}$.  Then we have $\p\in\Psupp^t_R(M)$, and hence $H^{t-\dim (R/\p)}_{\p R_{\p}}(M_{\p})\neq 0$.  Because $\p\notin\Psupp^i_R(M)$ for all $i>t$, we have that $H^{i-\dim (R/\p)}_{\p R_{\p}}(M_{\p})=0$ for all $i>t.$ It follows that  $\dim (M_{\p})=t-\dim (R/\p ).$

(ii). Let $\p\in\nCM(M).$ Then  $\depth (M_{\p})<\dim (M_{\p}).$ By (i) we obtain  $k<t$, where $k=\min\{i\mid \p\in\Psupp^i_R(M)\}$ and $t=\max\{i\mid \p\in\Psupp^i_R(M)\}$.  So, we have $k<t$ and  $\p\in\Psupp^k_R(M)\cap \Psupp^t_R(M)$. Conversely, if $\p\in \Psupp^i_R(M)\cap \Psupp^j_R(M)$ for some $0\leqslant i<j\leqslant d$ then $\depth (M_{\p})\leqslant i-\dim (R/\p)<j-\dim (R/\p)\leqslant \dim (M_{\p})$ by (i), and hence $\p\in\nCM(M).$  

(iii)  Let $\displaystyle \p\in\bigcup_{i\leqslant s}\Psupp^i_R(M).$  Then $\p\in\Psupp^r_R(M)$ for some $r\leqslant s.$ Set $$k=\min\{i\mid \p\in\Psupp^i_R(M)\}.$$ Then $k\leqslant r\leqslant s.$ Therefore we have by (i) that  $$\depth (M_{\p})+\dim (R/\p)=(k-\dim (R/\p))+\dim (R/\p)=k\leqslant s.$$ 
Conversely, let $\p\in\Supp_R(M)$ such that  $\depth (M_{\p})+\dim (R/\p)\leqslant s.$ If $\displaystyle \p\notin\bigcup_{i\leqslant s}\Psupp^i_R(M)$ then $\depth (M_{\p})>s-\dim (R/\p)$ by (i), i.e. $\depth (M_{\p})+\dim (R/\p)>s$, a contradiction.

(iv)  Assume that  $\displaystyle \p\notin \bigcup_{i<d}\Psupp^iM$. Then we get  by (iii)  that $\depth (M_{\p})+\dim (R/\p)=d$. Therefore  $M_{\p}$ is Cohen-Macaulay of dimension $d-\dim (R/\p).$
\end{proof}

\begin{corollary} \label{C:2a} Suppose that $M$ is equidimensional and the  ring $R/\Ann_RM$ is  catenary. Then  $\Psupp^i_R(M)$ is closed for $i=0,1,d$ and  $\displaystyle \nCM (M)=\bigcup_{i=0}^{d-1} \Psupp^i_R(M).$  
\end{corollary}

\begin{proof}   It is clear that $\Psupp^0_R(M)\subseteq\{\m\}$, hence $\Psupp^0_R(M)$ is closed. Let $\p\in\Psupp^1(M),$ then $\dim (R/\p)\leqslant 1.$ If $\dim (R/\p)=1$ then  $H^0_{\p R_{\p}}(M_{\p})\neq 0$ and hence $\p\in \Ass_RM.$ So, $\Psupp^1_R(M) \subseteq\{\m\}\cup \{\p\in\Ass M\mid \dim (R/\p)=1\}.$ Hence $\Psupp^1(M)$ has finitely many minimal elements. Since $R/\Ann_RM$ is  catenary, $\Psupp^1(M)$ is closed under specialization by [BS1, Lemma 2.2]. It follows that $\Psupp^1_R(M)$ is closed.   As $M$ is equidimensional, we have $$\Var(\Ann_R M)=\underset{\p\in\Ass M, \dim (R/\p)=d}{\bigcup}\Var (\p)=\Var (\Ann_RH^d_{\m}(M)).$$   Since $R/\Ann_RM$ is catenary, $\Psupp^d_R(M)=\Var (\Ann_RM)$ by [NA, Corollary 3.4(iv)]. Therefore $\Psupp^d_R(M)$ is closed and  $\Psupp^i_R(M)\cap \Psupp^d_R(M)=\Psupp^i_R(M)$ for all $i=1,\ldots ,d-1.$  Now the result follows by Theorem \ref{T:1a}(ii).
\end{proof}

 The following result, which is an immediate consequence of Theorem \ref{T:1a}(ii), gives a sufficient condition for the non Cohen-Macaulay locus of $M$ to be closed.  

\begin{corollary} \label{C:1} If $\Psupp^i_R(M)$ is closed for all $i\leqslant d$ then $\nCM (M)$ is closed.
\end{corollary}

It is natural to ask if the  converse statement of Corollary \ref{C:1} is true. Here is an answer for the case where $M$ is equidimensional  of dimension $3$. 

\begin{corollary} \label{C:2} Suppose that $M$ is equidimensional and $\dim M=3$.   If $R/\Ann_RM$ is  catenary then  $\Psupp^i_R(M)$ is closed for all $i\neq 2$ and $\displaystyle \nCM (M)=\bigcup_{i=0}^2\Psupp^i_R(M).$ In this case, $\nCM (M)$ is closed if and only if $\Psupp^2_R(M)$ is closed. 
\end{corollary}

\begin{proof}   By Corollary \ref{C:2a},  $\Psupp^i_R(M)$ is closed for all $i\neq 2$ and  $\displaystyle \nCM (M)=\bigcup_{i=0}^2\Psupp^i_R(M).$ So,  if $\Psupp^2_R(M)$ is closed then so is $\nCM(M).$ Assume $\Psupp^2_R(M)$ is not closed.   Since the ring $R/\Ann_RM$ is catenary, $\Psupp^2_R(M)$ is closed under specialization by [BS1, Lemma 2.2].  As $\Psupp^2_R(M)$ is not closed, it has infinitely many minimal elements. Note that $1\leqslant \dim(R/\p)\leqslant 2$ for all $\p\in\min\Psupp^2_R(M)$ by Lemma \ref{L:1d}(i), and $\dim (R/\p)\leqslant 1$ for all $\p\in\Psupp^1_R(M)\cup \Psupp^0_R(M).$  So,  each minimal element of $\Psupp^2_R(M)$ is  minimal in $\nCM(M).$ Therefore $\nCM(M)$ has infinitely many minimal elements and hence it is not closed. \end{proof}

 By M. Brodmann and R. Y. Sharp [BS1, Example 3.1],  there exists a  Noetherian local domain $(R,\m)$ of dimension $3$  such that $R$ is universally catenary,  $\Psupp^2(R)$ is not closed, and the non Cohen-Macaulay locus of $R$ is not  closed.

In case  the ring $R/\Ann_RM$ is not catenary, the converse statement of Corollary \ref{C:1} is not true. Before giving an example, we need the following result.

\begin{corollary} \label{C:3} Suppose that $\dim M=3$ and  $\dim (R/\p )=3$ for all $\p\in\Ass_R M$.  Assume that the ring $R/\Ann_RM$ is not catenary. Then $\Psupp^3_R(M)$ is not closed. Moreover,  $\Psupp^0_R(M)=\emptyset$, $\Psupp^1_R(M)\subseteq\{\m\}$ and
 $$\nCM (M)=\Psupp^2_R(M)\cap \Psupp^3_R(M).$$ 
\end{corollary}

\begin{proof}  Since $R/\Ann_R M$ is not catenary and $M$ is equidimensional, $\Psupp^3_R(M)$ is not closed  by [NA, Corollary 3.4(iv)].  It is clear that $\Psupp^0_R(M)=\emptyset$ and $\Psupp^1_R(M)\subseteq\{\m\}.$ So, by Theorem \ref{T:1a}, to show that  $\nCM (M)=\Psupp^2_R(M)\cap \Psupp^3_R(M),$ it is enough to check that $\m\in  \Psupp^2_R(M)\cap \Psupp^3_R(M).$ As $H^3_{\m}(M)\neq 0$, we get $\m\in\Psupp^3_R(M).$   As $R/\Ann_RM$ is not catenary, there exists $\p\in\Ass_RM$ such that $R/\p$ is  a non catenary domain of dimension $3$. Therefore by the same arguments as in the proof of \cite[Proposition 4.6(iv)]{CDN}, there exists $\widehat\p\in\Ass_{\R} (\R/\p\R)$ such that $\dim (\R/\widehat\p)=2.$ Since $\displaystyle \Ass_{\R} \widehat M=\bigcup_{\q\in\Ass_R M}\Ass (\R/\q\R)$ by \cite[Theorem 23.2(ii)]{Mat}, it follows that $\widehat\p\in\Ass_{\R}\widehat M.$ So $\widehat\p\in\Att_{\R}(H^2_{\m\R}(\widehat M))$ by \cite[11.3.3]{BS}.   Hence $H^2_{\m}(M)\neq 0$. So $\m\in\Psupp^2_R(M).$
\end{proof}

Now we give an example to show that the converse statement of Corollary \ref{C:1} is not true.
\begin{example} There exists a Noetherian local domain $R$ of dimension $3$ such that the non Cohen-Macaulay locus of $R$ is closed, but $\Psupp^2(R)$ and $\Psupp^3(R)$ are not closed.  
\end{example} 
\begin{proof} It follows by [BS1, Example 3.2] that there exists a Noetherian local domain $(R,\m)$ of dimension $3$  such that $R$ is not catenary, $\Psupp^2(R)$ and $\Psupp^3(R)$ are not closed and  \begin{align}\Psupp^2(R)\setminus\{\m,0\}&=\{\p\in\Spec R\mid \docao (\p)+\dim (R/\p)=2\},\notag\\
\Psupp^3(R)&=\{\p\in\Spec R\mid \docao (\p)+\dim (R/\p)=3\}.\notag
\end{align}
Therefore we get by Corollary \ref{C:3} that $$\nCM(M)=\Psupp^2(R)\cap \Psupp^3(R)\subseteq \{\m, 0\}.$$
It is clear that $0\notin\Psupp^2(R).$  As $R$ is not catenary, $R$ is not Cohen-Macaulay. Therefore $\nCM(M)=\{\m\}$ which is closed.
\end{proof}

\section{A connection with the universal catenarity and unmixedness}

From now on, for each integer $i$ we set $\frak a_i(M)=\Ann_RH^i_{\m}(M)$. Set $$\frak a(M)=\frak a_0(M)\frak a_1(M)\ldots \frak a_{d-1}(M).$$

Let $i$ be an integer. By Lemma \ref{L:1b}, $\Psupp^i_RM$ is a  subset of $\Var (\frak a_i(M))$.  Here is a criterion for these sets to be the same. 

\begin{lemma} \label{L:1c} {\rm \cite[Proposition 2.5]{BS1}} If the ring $R/\Ann_RM$ is universally catenary and all its formal fibres are Cohen-Macaulay then  $\Psupp^i_RM=\Var (\frak a_i(M))$ for all integers $i$.
\end{lemma}

  Using Lemma \ref{L:1c} together with the facts in Sections 2,3, we have the following properties, some of them have already known in case where the ring $R$ is a quotient of a Gorenstein local ring, cf.  \cite[Proposition 3.8]{Sh}, \cite[Corollaries 3,6]{Sch1}, \cite[Theorem 1.2]{C}. 

\begin{corollary}\label{C:6} Let $i\geq 0$ be an integer.  Suppose that $R/\Ann_RM$ is universally catenary and all its formal fibres are Cohen-Macaulay. Then we have

(i) $\dim (R/\p)\leqslant i$ for all $\p\in\Var(\frak a_i(M))$.

(ii) $(\Att_R H^i_{\m}(M))_i=(\Var (\frak a_i(M)))_i=(\Ass_RM)_i.$

(iii) $\psd^i(M)=\psd^i(\widehat M)=\dim (R/\frak a_i(M)).$

(iv) $\nCM(M)$ is closed.

 (v) If $M$ is equidimensional then $\displaystyle \nCM (M)=\Var (\frak a(M)).$ 

\end{corollary}

Following M. Nagata \cite{Na}, we say that  $M$ is {\it quasi unmixed} if $\widehat M$ is equidimensional, i.e. $\dim (\R/\widehat\p)=d$ for all $\widehat\p\in\min\Ass_{\R}\widehat M.$  We say that $M$ is {\it unmixed} if $\dim (\R/\widehat\p)=d$ for all $\widehat\p\in\Ass_{\R}\widehat M.$ 

\begin{theorem}\label{T:2} Set $\displaystyle T(M)=\bigcup_{0\leqslant i<j\leqslant d}\Var (\frak a_i(M)+\frak a_j(M)).$ The following statements are true.

(i)  If the ring $R/\Ann_RM$ is universally catenary and all its formal fibres are Cohen-Macaulay then $\nCM (M)=T(M).$ 

(ii) If $\nCM (M)=T(M)$ then the ring $R/\Ann_RM$ is universally catenary and $R/\p$ is unmixed for all $\p\in\min\Ass_R M.$
\end{theorem}

\begin{proof}  (i) follows by  Theorem \ref{T:1a}(ii) and Lemma \ref{L:1c}.  

(ii). Let $\p\in\min\Ass_RM$.  Set $\dim (R/\p)=t.$ Assume that  $R/\p $ is not unmixed. Then there exists  $\widehat \p\in\Ass (\widehat R/\p\widehat R)$ such that $\dim (\R/\widehat\p) =k<t.$  Note that $k<t\leqslant d$. We have by [Mat, Theorem  23.2(ii)] that $$\Ass\widehat M=\bigcup_{\q\in\Ass M}\Ass (\R/\q\R ).$$ Therefore  $\widehat \p\in\Ass \widehat M.$ Since $\dim (\R/\widehat\p)=k,$ we get by [BS, 11.3.9] that $\widehat\p\in\Att_{\R}(H^k_{\m\R}(\widehat M))$. Note that $H^k_{\m\R}(\widehat M)\cong H^k_{\m}(M)$ as $\R$-modules. So, $\widehat\p\in\Att_{\R}(H^k_{\m}(M))$. By \cite[8.2.4, 8.2.5]{BS},  $\p=\widehat\p\cap R\in\Att_R(H^k_{\m}(M))$. So $\frak a_k(M)\subseteq \p$ by Lemma \ref{L:1a}. Moreover, since $\dim (R/\p)=t$ and $\p\in\Ass M,$ it follows by [BS, 11.3.9] that $\p\in\Att_R(H^t_{\m}(M)).$ Hence $\frak a_t(M)\subseteq \p$ by Lemma \ref{L:1a}. Therefore we have $\p\in\Var (\frak a_k(M)+\frak a_t(M))$, where $k<t\leqslant d.$ So $\p\in T(M)=\nCM (M)$ by the hypothesis.  Since $\p\in\min\Ass_RM$, it follows that $M_{\p}$ is of finite length, and therefore $M_{\p}$ is Cohen-Macaulay. This is a contradiction. Thus $R/\p$ is unmixed for all $\p\in\min\Ass_RM.$

 Finally,  we show that $R/\Ann_R M$ is universally catenary. By [Mat, Theorem 31.7, (1)$\Leftrightarrow$(2)],  it is sufficient to show that $R/\p$ is quasi unmixed for all prime ideal $\p$ of $R$ containing $\Ann_RM$. Let $\p$ be a prime ideal containing $\Ann_RM.$ Then there exists $\q\in\min\Ass_RM$ such that $\q\subseteq \p .$ By the above  fact that we have just proved,  the domain $R/\q$ is unmixed. Therefore,  it follows by [Mat, Theorem 31.6,(ii)] that $R/\p$ is quasi unmixed. 
\end{proof}
 
 Let $\p\in\Spec (R).$ M. Nagata \cite{Na1} asked  whether $R/\p$ is unmixed  provided $R$ is  unmixed. However, M. Brodmann and C. Rotthaus \cite{BR} gave a counterexample to the question by Nagata. Therefore it is natural to ask   under which condition, $R/\p$ is unmixed.  Before giving a partial answer  to this question, we need the following result concerning to  the Serre conditions on $M$.  Let $r>0$ be an integer.  $M$ satisfies the {\it  Serre condition $(S_r)$} if 
$$\depth (M_{\p})\geq \min\{r,\dim (M_{\p})\}\ \text{for all}\ \p\in\Supp_R(M).$$ 

\begin{lemma} \label{L:5} Let $r\geq 0$ be an integer.  Assume that $M$ is equidimensional and the ring $R/\Ann_RM$ is catenary. Then $M$ satisfies  $(S_r)$ if and only if $\psd^i(M)\leqslant i-r$ for all $i<d.$ In particular, if $M$ satisfies  $(S_r)$ then $\dim (R/\p)\leqslant d-r-1$ for all $\p\in\nCM(M).$
\end{lemma}

\begin{proof} Assume that $M$ satisfies the Serre condition  $(S_r)$.   Suppose that  there exists an integer $n<d$ such that $\psd^n(M)> n-r$. Let $\p\in\Psupp^n_R(M)$ such that $\dim (R/\p )=\psd^n(M).$ Then  $\depth (M_{\p})+\dim (R/\p)\leqslant n<d$ by Theorem \ref{T:1a}(iii). Therefore we get by Theorem \ref{T:1a}(i) that 
 $$\depth (M_{\p})\leqslant n-\dim (R/\p)=n-\psd^n(M)<n-(n-r)=r.$$ As $M$ satisfies the Serre condition  $(S_r)$, it follows that $\depth (M_{\p})=\dim (M_{\p}).$ Since $M$ is equidimensional and the ring $R/\Ann_RM$ is catenary, we have $$\depth (M_{\p})+\dim (R/\p )=\dim (M_{\p})+\dim (R/\p )=d,$$ this is a contradiction. 

Assume that  $\psd^i(M)\leqslant i-r$ for all $i<d.$ Let $\p\in\Supp_R(M).$ If $M_{\p}$ is Cohen-Macaulay then there is nothing to do.  So assume that $\p\in\nCM (M).$ Then we have by Corollary \ref{C:2a} that $\displaystyle \p\in\bigcup_{i=0}^{d-1}\Psupp^i_R(M).$  Set $k=\min\{i\mid \p\in\Psupp^i_R(M)\}$, then $k<d$ and $\p\in\Psupp^k_R(M).$   By the hypothesis, $\dim (R/\p )\leqslant \psd^k(M)\leqslant k-r.$ So, by Theorem \ref{T:1a}(i), $$\depth (M_{\p})=k-\dim (R/\p )\geq k-(k-r)=r.$$
\end{proof}

It is known that if $R$ is a quotient of a Gorenstein ring then $M$ satisfies the Serre condition  $(S_r)$ if and only if $\dim (R/\frak a_i(M))\leqslant i-r$ for all $i<d,$ cf. \cite[Lemma 3.2.1]{Sch}.  Here, Lemma  \ref {L:1c} and Lemma \ref{L:5} show that  this result is still true when  $R/\Ann_RM$ is universally catenary and all its formal fibres are Cohen-Macaulay. 

\begin{theorem} \label{T:3} Let $r\geq 1$ be an integer.  Suppose $M$ is equidimensional and $M$ satisfies the  Serre condition  $(S_r)$. If $\nCM(M)=\Var (\frak a(M))$ then $R/\p$ is unmixed for all $\p\in\Supp_R(M)$ with $\dim (R/\p)\geq d-r.$
\end{theorem}

\begin{proof} We prove by induction on $r$.  Let $r=1.$    Set $\displaystyle T(M)=\bigcup_{0\leqslant i<j\leqslant d}\Var \big(\frak a_i(M)+\frak a_j(M)\big).$
We have by Theorem \ref{T:1a}(ii) and Lemma \ref{L:1b} that
$$\nCM(M)=\bigcup_{0\leqslant i<j\leqslant d}\big(\Psupp^i_R(M)\cap \Psupp^j_R(M)\big)\subseteq T(M)\subseteq \Var (\frak a(M)).$$ As $\nCM(M)=\Var (\frak a(M))$ by the hypothesis,  $\displaystyle \nCM(M)=T(M).$ By Theorem \ref{T:2}(ii) , the ring $R/\Ann_RM$ is universally  catenary. Since $M$ is equidimensional and $M$ satisfies the  Serre condition $(S_1)$, it follows by Lemma \ref{L:5} that $\dim (R/\q)\leqslant  d-2$ for all $\q\in\nCM(M).$ 
Let $\p\in\Supp_R(M)$ such that $\dim (R/\p)\geq d-1.$ If $\dim (R/\p)=d$ then $\p\in\min\Ass_R(M)$ and hence $R/\p$ is unmixed by Theorem \ref{T:2}(ii).  Let $\dim (R/\p)=d-1.$ As $M$ satisfies the Serre condition  $(S_1)$, it follows that $\Ass_RM=\min\Ass_RM.$ Hence $\dim (R/\q)=d$ for all $\q\in\Ass_RM.$ So, there exists  $x\in \p$ such that $x$ is $M$-regular. Assume that $R/\p$ is not unmixed. Then there exists $\widehat\p\in\Ass (\R/\p\R)$ such that $\dim (\R/\widehat\p)=k<d-1.$ Since $x\in \p$ and  $\dim (R/\p)=\dim (M/xM)$, we get $\p\in\min(\Ass_R(M/xM))$.  By [Mat, Theorem 23.2,(ii)], $$\Ass_{\R}(\widehat M/x\widehat M)=\bigcup_{\q\in\Ass_R(M/xM)}\Ass (\R/\q\R).$$ Therefore $\widehat\p\in\Ass_{\R}(\widehat M/x\widehat M).$ From the exact sequence $0\lr M\overset{x}{\lr}M\lr M/xM\lr 0$, we get the exact sequence
$$\ \ \ \ \ 0\lr H^k_{\m}(M)/xH^k_{\m}(M)\lr H^k_{\m}(M/xM)\lr 0:_{H^{k+1}_{\m}(M)}x\lr 0.\ \ \ \ \ (1)$$
   As $\dim (\R/\widehat\p)=k,$ we have  by Corollary \ref{C:6}(ii) that $\widehat\p\in\Att_{\R}(H^k_{\m\R}(\widehat M/x\widehat M))$. By Lemma \ref{L:1a} we have $\widehat\p\in \Var (\Ann_{\R} H^k_{\m}(\widehat M/x\widehat M))$. Hence $\p=\widehat\p\cap R\in\Var (\Ann_RH^k_{\m}(M/xM)).$ Therefore,  from  the exact sequence (1) we have $\p\in\Var (\frak a_k(M))\cup \Var (\frak a_{k+1}(M)).$ Since $k<d-1,$ we get $\p\in\Var (\frak a(M)).$ So, $\p\in\nCM(M)$ by our hypothesis.  This is a contradiction since $\dim (R/\p)=d-1.$ Thus, the result is true for $r=1.$

  Let $r>1$ and assume that the result is true for all finitely generated equidimensional $R$-modules $L$ satisfying the Serre condition  $(S_{r-1})$ such that $\nCM(L)=\Var (\frak a(L)).$ Let $\p\in\Supp_R(M)$ such that $\dim (R/\p)\geq d-r.$ If $\dim (R/\p)=d$ then $R/\p$ is unmixed by Theorem \ref{T:2}(ii). Assume that $\dim (R/\p)<d.$
As $\dim (R/\q)=d$ for all $\q\in\Ass_R(M),$ it follows that $\p\not\subseteq\q$ for all $\q\in\Ass_R(M).$ Therefore there exists an $M$-regular element $x\in \p.$  Let $\q\in\Supp_R(M/xM).$ As $M$ satisfies the Serre condition  $(S_r)$ and $x$ is an $M_{\q}$-regular element, it follows that 
$$\depth (M/xM)_{\q}=\depth (M_{\q})-1\geq\min\{\dim (M/xM)_{\q}, r-1\}.$$
 Therefore $M/xM$ satisfies the Serre condition   $(S_{r-1}).$  Let  $\q\in\min\Ass_R(M/xM).$ Then $\depth (M_{\q})=1.$ Note that  $M$ satisfies the Serre condition  $(S_2)$ since $r>1$.  Therefore $\dim (M_{\q})=1.$  Since $M$ is equidimensional and $R/\Ann_RM$ is catenary, it follows that $\dim (R/\q)=d-1.$ Thus, $M/xM$ is equidimensional.  By Theorem \ref{T:1a}(iv) and Lemma \ref{L:1b},  $\nCM(M/xM)\subseteq \Var (\frak a(M/xM)).$ Let $\q\in \Var (\frak a(M/xM)).$ Then $\q\in\Var (\frak a_k(M/xM))$ for some $k<d-1.$ Therefore we get from the exact sequence (1) that $\q\in\Var (\frak a_k(M))$ for some $k<d.$ Hence $\q\in\Var (\frak a(M)).$ So we have by hypothesis that $\q\in\nCM(M),$ i.e. $M_{\q}$ is not Cohen-Macaulay. Since $x$ is $M_{\q}$-regular, $M_{\q}/xM_{\q}$ is not Cohen-Macaulay. Therefore $(M/xM)_{\q}$ is not Cohen-Macaulay, i.e. $\q\in\nCM (M/xM).$ So, $$\nCM(M/xM)=\Var (\frak a(M/xM)).$$  Note that $\p\in\Supp_R(M/xM)$ and $$\dim (R/\p)\geq d-r=\dim (M/xM)-(r-1).$$ Now by induction applying to $M/xM$, the ring $R/\p$ is unmixed.
\end{proof}

Finally, we present an example to clarify some results in this paper. 
 
\begin{example} \label{E:1} Let $(R,\m)$ be the Noetherian local domain  of dimension $3$ constructed by M. Brodmann and C. Rotthaus \cite{BR}     such that $\R$ is a domain and $R/\p$ is not unmixed for some $\p\in\Spec (R).$  Let $\p$ be such a prime ideal. Then $\p\in\Var (\frak a(R))\setminus \nCM(R).$ 
\end{example}
\begin{proof}  It is easy to check that $\dim (R/\p)=2$ and there exists $\widehat\p\in\Ass (\R/\p\R)$ such that $\dim (\R/\widehat\p)=1.$    By the same arguments as in the proof of Theorem \ref{T:3}, it follows that  $\p\in\Var (\frak a_1(R))\cup \Var (\frak a_2(R)).$  Hence $\p\in\Var (\frak a(R)).$ As $\dim (R/\p)=2$, we have $\dim (R_{\p})=1=\depth (R_{\p}).$ Therefore $\p\notin\nCM(R).$
\end{proof}

Let $R$ be the Noetherian local domain of dimension $3$ as in Example \ref{E:1}.  Since $\R$ is a domain, $R$ universally catenary. Since $\Var (\frak a(R))\neq \nCM(R)$, it shows that the converse statement of Theorem \ref{T:2}(ii) is not true. Moreover,  $R$ is equidimensional and $R$ satisfies the Serre condition $(S_1),$ but $R/\p$ is not unmixed for some $\p$ of dimension $2.$ Therefore the hypothesis $\nCM(M)=\Var (\frak a(M))$ in Theorem 4.5 can not be removed.  

\medskip
\noindent{\bf Acknowledgment}. The authors thank for the referee's useful suggestions.

\end{document}